\documentclass[notitlepage]{scrartcl}
\usepackage{amsmath}
\usepackage{amssymb}
\usepackage{amsthm}
\usepackage{abstract}
\usepackage{xcolor}
\usepackage{hyperref}
\hypersetup{
    colorlinks=true,
    linkcolor=blue,
    filecolor=magenta,      
    urlcolor=cyan
    }
\newtheorem{theorem}{Theorem}
\newtheorem{thm}{Theorem}[section]

\newtheorem{lemma}[thm]{Lemma}
\newtheorem{claim}[thm]{Claim}
\newtheorem{remark}[thm]{Remark}

\usepackage[margin=1in]{geometry}
\title{Expansion in Supercritical Random Subgraphs of Expanders and its Consequences}
\newcommand{\footremember}[2]{%
    \footnote{#2}
    \newcounter{#1}
    \setcounter{#1}{\value{footnote}}%
}
 
\author{%
Sahar Diskin \footremember{alley}{School of Mathematical Sciences, Tel Aviv University, Tel Aviv 6997801, Israel. Email: sahardiskin@mail.tau.ac.il.}%
\and Michael Krivelevich \footremember{trailer}{School of Mathematical Sciences, Tel Aviv University, Tel Aviv 6997801, Israel. Email:
krivelev@tauex.tau.ac.il. Research supported in part by USA–Israel BSF grant 2018267.}%
}
\begin{document}
	\maketitle
	
\begin{onecolabstract}
    In 2004, Frieze, Krivelevich and Martin \cite{FKM} established the emergence of a giant component in random subgraphs of pseudo-random graphs. We study several typical properties of the giant component, most notably its expansion characteristics. We establish an asymptotic vertex expansion of connected sets in the giant by a factor of $\tilde{O}\left(\epsilon^2\right)$. From these expansion properties, we derive that the diameter of the giant is typically $O_{\epsilon}\left(\log n\right)$, and that the mixing time of a lazy random walk on the giant is asymptotically $O_{\epsilon}\left(\log^2 n\right)$. We also show similar asymptotic expansion properties of (not necessarily connected) linear sized subsets in the giant, and the typical existence of a large expander as a subgraph.
\end{onecolabstract}

\section{Introduction}
Percolation theory, the study of which was initiated by Broadbent and Hammersley \cite{broadbent and hammersley} in 1957, is a mathematical discipline researching the following probabilistic model: given a graph $G$, the percolated subgraph $G_p$ is the random subgraph of $G$ obtained by retaining each edge of $G$ independently with probability $p$. We refer the reader to \cite{kesten}, \cite{grimmet} and \cite{percolation book} for systematic coverage of percolation theory.

Percolation on several concrete base graphs has been studied extensively. One notable example is percolation on the complete graph $K_n$, which is equivalent to the classical model of the random graph $G(n,p)$. In their groundbreaking paper \cite{Erdos1} from 1960, Erd\H{o}s and R\'enyi discovered that $G(n,p)$ (working on its close analogue $G(n,m)$) undergoes a phase transition around probability $p=\frac{1}{n}$: for any constant $\epsilon>0$, if $p=\frac{1-\epsilon}{n}$ then all the connected components of $G(n,p)$ are typically of size $O(\log n)$, while for $p=\frac{1+\epsilon}{n}$, \textbf{whp}\footnote{With high probability, that is, with probability tending to $1$ as $n$ tends to infinity.} there emerges a unique component of linear size in $G(n,p)$, usually called the \textit{giant component}. In the same paper \cite{Erdos1}, Erd\H{o}s and R\'enyi also obtained the asymptotic order of this giant component. We refer the reader to \cite{JLR}, \cite{book random graphs} and \cite{intro to random graphs} for a systematic coverage of random graphs. Similar phenomenon of an emergence of a unique linear sized component has been discovered in percolation on other concrete base graphs, with one well studied example being the $d$-dimensional hypercube \cite{AKS, BKL}. 

When trying to show the emergence of a unique giant component in percolation on a large class of graphs $\mathcal{G}$, certainly some assumptions on $G\in\mathcal{G}$ have to be made. For example, letting $\mathcal{G}$ be the family of $d$-regular graphs, $G\in \mathcal{G}$ could be a collection of vertex disjoint cliques of size $d+1$, having no large connected components to begin with. It is thus necessary to impose some restrictions on the edge-distribution of the graphs. Hence, it may be natural to consider \textit{pseudo-random graphs}. 

In this paper, we are concerned with a specific model of pseudo-random graphs called $(n,d,\lambda)$-\textit{graphs} (also called spectral or algebraic expanders). Informally, pseudo-random graphs are graphs whose edge distribution resembles that of random graphs, with the same number of vertices and edge density. An $(n,d,\lambda)$-graph is a $d$-regular graph on $n$ vertices, where its eigenvalues (that is, the eigenvalues of its adjacency matrix) $d= \lambda_1 \ge \lambda_2 \ge\ldots \ge \lambda_n$ satisfy $\lambda=\max\left\{|\lambda_2|, |\lambda_n|\right\}$. The greater the ratio between $d$ and $\lambda$, also called the spectral ratio, the more tightly the distribution of the edges of the graph approaches that of a truly random graph $G\left(n,\frac{d}{n}\right)$. This follows from the expander mixing lemma, due to Alon and Chung \cite{EML} (stated as Lemma \ref{eml} in this paper). Thus, $(n,d,\lambda)$-graphs serve frequently as a standard model for pseudo-random graphs. Crucially, note that a large spectral ratio implies that $d$ is large as well. We refer to \cite{expanders} for a comprehensive survey on the subject of pseudo-random graphs.

Returning to the subject at hand, the emergence of a unique giant component in percolation on $(n,d,\lambda)$-graphs has been shown in 2004 by Frieze, Krivelevich, and Martin \cite{FKM}. For such graphs, it was shown in \cite{FKM} that around $p=\frac{1}{d}$ a phase transition occurs, with respect to the sizes of the components, having many properties similar to the phase transition of $G(n,p)$ around $p=\frac{1}{n}$.

To be more concrete, let $y=y(\epsilon)$ be the unique solution in $(0,1)$ of:
\begin{align}
y\exp(-y)=(1+\epsilon)\exp\left(-(1+\epsilon)\right).
\end{align}
Theorem 1 of $\cite{FKM}$ then states the following:
\begin{thm}[\cite{FKM}] \label{FKM-thm} \textit{Let $\epsilon>0$ be a small enough constant and let $\delta>0$ be such that $\delta\le \epsilon^4$. Let $G=(V,E)$ be an $(n,d,\lambda)$-graph with $\frac{\lambda}{d}\le \delta$. Let $p=\frac{1+\epsilon}{d}$. Form $G_p$ by retaining every edge of $G$ independently with probability $p$. Then, \textbf{whp}, there is a unique giant component $L_1$ of asymptotic size $\left(1-\frac{y}{1+\epsilon}\right)n$, where $y$ is defined as in $(1)$. \textbf{Whp}, all the other components are of size $O\left(\log n\right).$}
\end{thm}
Some comments are in order here. In the same paper \cite{FKM}, the authors showed that when $p=\frac{1-\epsilon}{d}$ (for a small constant $\epsilon>0$), then typically all the connected components are of order $O_{\epsilon}(\log n)$. Furthermore, this phase transition and the asymptotic order of the giant component are analogous, both quantitatively and numerically, to the classical phase transition in $G(n,p)$ around $p=\frac{1}{n}$. We thus call the regime where $p=\frac{1-\epsilon}{d}$ the \textit{subcritical regime} and the regime where $p=\frac{1+\epsilon}{d}$ the \textit{supercritical regime}. Observe that from $(1)$ and Theorem \ref{FKM-thm}, it can be seen that \textbf{whp}, $|L_1|=\left(1+o_{\epsilon}(1)\right)2\epsilon n$. Furthermore, we note that the requirement stated in \cite{FKM} is that $\delta\to 0$, but it can be easily verified that the proofs go through with $\delta\le \epsilon^4$, as stated above.

The giant components in the aforementioned percolated concrete graphs (such as $G(n,p)$ and the percolated hypercube) are known to be typically far from tree-like, and in fact are fairly rich combinatorial structures, with many interesting properties to study. In the case of $G(n,p)$, there has been extensive research into several properties of the giant, such as the typical diameter of (\cite{diameter2, diameter}), and the asymptotic mixing time of a random walk (\cite{FR2, BKW}) on the giant component. One approach taken in order to study such properties is to examine the expansion properties of the giant component. For example, this has been done in supercritical $G(n,p)$ by Benjamini, Kozma, and Wormald \cite{BKW} in order to study the mixing time of a random walk on the giant, and very recently by Erde, Kang and Krivelevich \cite{EKK} in supercritical percolation on the hypercube in order to obtain asymptotic bounds on the diameter and the mixing time of a random walk on the giant. Krivelevich, Reichman, and Samotij \cite{perturbed} studied the expansion properties of perturbed connected graphs, and obtained from them asymptotic bounds on the diameter and the mixing time. In \cite{expanders}, Krivelevich showed how to argue about expansion and to obtain several interesting properties of graphs from their expansion properties.

The aim of this paper is to study the typical expansion properties of the giant component of its random subgraph, and to derive from them several properties of the giant component, such as its asymptotic diameter and mixing time.

With this in mind, we are now able to state and discuss our results.

We start with some notation. Given a graph $G$, we let $N_G(S)$ be the external neighbourhood of $S\subseteq V(G)$ in $G$, and let $\partial_G(S)$ be the set of all the edges in $G$ with exactly one endpoint in $S$ (the edge boundary of $S$ in $G$). 

Our first main result regards the vertex-expansion properties:
\begin{theorem} \label{vertex-expansion}
\textit{Let $\epsilon>0$ be a small enough constant and let $\delta>0$ be such that $\delta\le \epsilon^4$. Let $G=(V,E)$ be an $(n,d,\lambda)$-graph with $\frac{\lambda}{d}\le \delta$. Let $p=\frac{1+\epsilon}{d}$. Form $G_p$ by retaining every edge of $G$ independently with probability $p$. Let $L_1$ be the largest component of $G_p$. Then, there exists an absolute constant $c>0$ such that \textbf{whp} for any $S\subseteq L_1$:
\begin{enumerate}
    \item if $\frac{16\ln n}{\epsilon^2}\le |S| \le \frac{\epsilon^2 n}{50}$ and $S$ spans a connected subgraph in $G_p$, then:
    \begin{align*}
        \bigg|N_{G_p}(S)\bigg|\ge\frac{c\epsilon^2|S|}{\ln\left(\frac{1}{\epsilon}\right)};
    \end{align*}
     and,
    \item if $\frac{\epsilon^2n}{50}\le |S| \le \frac{12\epsilon n}{11}$, then:
    \begin{align*}
        \bigg|N_{G_p}(S)\bigg|\ge\frac{c\epsilon^2|S|}{\ln^2\left(\frac{1}{\epsilon}\right)}.
    \end{align*}
\end{enumerate}}
\end{theorem}

Let us first note that it is necessary to require in the first part of the statement above that the sets of up to a small linear size are connected. Indeed, for simplicity, let us consider the classical model of $G(n,p)$. There it can be seen (for example, from the structural result of Ding, Lubetzky, and Peres \cite{DLP}) that typically one can find a non-connected subset of the giant component which is of a linear size, but expands by less than a $\tilde{O}\left(\epsilon^2\right)$ factor. 

Clearly, the vertex-expansion properties of Theorem \ref{vertex-expansion} imply that the same holds for edge-expansion as well. However, for linear sized (not necessarily connected) subsets, we can obtain a slightly better edge-expansion factor. Indeed, we have the following asymptotic edge-expansion property:
\begin{theorem} \label{edge-expansion}
\textit{Let $\epsilon>0$ be a small enough constant and let $\delta>0$ be such that $\delta\le \epsilon^4$. Let $G=(V,E)$ be an $(n,d,\lambda)$-graph with $\frac{\lambda}{d}\le \delta$. Let $p=\frac{1+\epsilon}{d}$. Form $G_p$ by retaining every edge of $G$ independently with probability $p$. Let $L_1$ be the giant component of $G_p$. Then, there exists an absolute constant $c>0$ such that \textbf{whp} for any $S\subseteq L_1$ such that $\frac{\epsilon^2 n}{50}\le |S|\le \frac{12\epsilon n}{11}$, we have 
    $$\bigg|\partial_{G_p}(S)\bigg|\ge\frac{c\epsilon^2|S|}{\ln\left(\frac{1}{\epsilon}\right)}.$$
}
\end{theorem}

Let us mention a few related results. In 2007, Ofek \cite{ofek} studied the edge-expansion property of $G_p$ when the spectral ratio is nearly optimal and $p\ge \frac{C}{\sqrt{d}}$, obtaining a typical edge-expansion by an $\Omega\left(\frac{1}{\ln n}\right)$ factor. In supercritical $G(n,p)$, Fountoulakis and Reed \cite{FR2}, and independently Benjamini, Kozma, and Wormald \cite{BKW} showed that asymptotically the edge-expansion of connected subsets of the giant is by an $\Omega(\epsilon)$ factor. Krivelevich \cite{firstexpanders, expanders} showed a method for finding large expanders as subgraphs of general graphs, and in particular the typical existence of a linearly large constant factor expander in supercritical $G(n,p)$. Theorems \ref{vertex-expansion} and \ref{edge-expansion} here establish \textbf{whp} an $\tilde{\Omega}\left(\epsilon^2\right)$ factor.

From Theorem \ref{vertex-expansion}, we are able to derive that the largest connected component of $G_p$ \textbf{whp} contains a large expander as a subgraph. Indeed, a simple corollary of Property 2 of Theorem \ref{vertex-expansion} is the following:
\begin{theorem} \label{expanding-subgraph}
\textit{Let $\epsilon>0$ be a small enough constant and let $\delta>0$ such that $\delta\le \epsilon^4$. Let $G=(V,E)$ be an $(n,d,\lambda)$-graph with $\frac{\lambda}{d}\le \delta$. Let $p=\frac{1+\epsilon}{d}$. Form $G_p$ by retaining every edge of $G$ independently with probability $p$. Let $L_1$ be the giant component of $G_p$. Then, \textbf{whp} there exists $L_1'\subseteq L_1$ with $|L_1'|\ge\frac{7\epsilon n}{4}$ such that for every $S\subseteq L_1'$ with $|S|\le \frac{|L_1'|}{2}$,
$$\bigg|N_{L_1'}(S)\bigg|\ge\frac{c \epsilon^2|S|}{\ln^2\left(\frac{1}{\epsilon}\right)}.$$}
\end{theorem}

\begin{remark}
The fraction $\frac{7}{4}$ in Theorem \ref{expanding-subgraph} can be replaced by any constant strictly smaller than $2$, while adjusting the statements and proofs of the other theorems as well. It is left as such for ease of presentation.
\end{remark}

From the typical existence of a large, relatively good expander as a subgraph, one can obtain several other interesting properties of the giant component, such as the existence \textbf{whp} of a large complete minor, a long cycle and cycles of different lengths. We refer the reader to \cite{expanders} for a survey including many results of this type, and to \cite{HLW} for an extensive survey on expanders and their applications.

From the expansion properties of the giant component, we will be able to derive an asymptotic bound on its diameter:
\begin{theorem} \label{diameter}
\textit{Let $\epsilon>0$ be a small enough constant and let $\delta>0$ be such that $\delta\le \epsilon^4$. Let $G=(V,E)$ be an $(n,d,\lambda)$-graph with $\frac{\lambda}{d}\le \delta$. Let $p=\frac{1+\epsilon}{d}$. Form $G_p$ by retaining every edge of $G$ independently with probability $p$. Let $L_1$ be the largest component of $G_p$. Then, there exists an absolute constant $C$ such that \textbf{whp} the diameter of $L_1$ is at most $$C\frac{\ln \left(\frac{1}{\epsilon}\right)\ln n}{\epsilon^2}.$$}
\end{theorem}

For $G(n,p)$ with $p=\frac{1+\epsilon}{n}$, the diameter of the giant component is known to be typically $\Theta\left(\frac{\ln n}{\epsilon}\right)$ (see the results of Fernholz and Ramachandran \cite{diameter2} and of Riordan and Wormald \cite{diameter}). Theorem \ref{diameter} is within an $\tilde{\Omega}\left(\frac{1}{\epsilon}\right)$ factor from that.

In \cite{FR2}, Fountoulakis and Reed obtained a good control over the typical edge-expansion of connected subsets of the giant component in $G(n,p)$ and utilised it in order to bound the mixing time of a random walk on the giant component. Adjusting their method to our setting, we obtain:
\begin{theorem} \label{mixing-time}
\textit{Let $\epsilon>0$ be a small enough constant and let $\delta>0$ be such that $\delta\le \epsilon^4$. Let $G=(V,E)$ be an $(n,d,\lambda)$-graph with $\frac{\lambda}{d}\le \delta$. Let $p=\frac{1+\epsilon}{d}$. Form $G_p$ by retaining every edge of $G$ independently with probability $p$. Let $L_1$ be the largest component of $G_p$. Then, there exists an absolute constant $C$ such that \textbf{whp} the mixing time of the lazy random walk on $L_1$ is at most
\begin{align*}
    C\frac{\ln^2n}{\epsilon^4}.
\end{align*}}
\end{theorem}

For the supercritical $G(n,p)$, Fountoulakis and Reed \cite{FR2}, and independently Benjamini, Kozma and Wormald \cite{BKW} established the asymptotic order of magnitude of the mixing time of a random walk on the giant component is asymptotically $\Theta\left(\log^2n\right)$. Thus, the order of magnitude (in $n$) of our estimate in Theorem \ref{mixing-time} is optimal.

In Section $2$ we state and establish several lemmas which we will use throughout the proofs, as well as the key lemma for proving Theorems \ref{vertex-expansion} and \ref{edge-expansion} (Lemma \ref{BFS}), utilising a random variant of the Breadth First Search algorithm. In Section $3$ we prove Theorems \ref{vertex-expansion}, \ref{edge-expansion}, and \ref{expanding-subgraph}. Theorem \ref{diameter} is proven in Section $4$, and Theorem \ref{mixing-time} is proven in Section $5$.

\section{Auxiliary Lemmas}
We denote by $e_G(A,B)$ the number of edges in $G$ with one endpoint in $A$ and the other endpoint in $B$. When $A=B$, we write for short $e_G(A):=\frac{e_G(A,A)}{2}$. Note that $e_G(A)$ is the number of edges inside $A$. When the graph we refer to is obvious, we sometimes omit the subscript.

We begin by stating a couple of lemmas which we will use throughout the proofs. The first one is the famed \textit{expander mixing lemma} due to Alon and Chung \cite{EML} (which we state in the same form as Theorem 2.11 of \cite{pseudo-survey}), which gives us control over the edge-distribution of the graph through its spectral ratio:
\begin{lemma}[\cite{EML}] \label{eml}
\textit{Let $G=(V,E)$ be an $(n,d,\lambda)$-graph. Then, for every two subsets $A,B\subseteq V$, $$\Bigg|e(A,B)-\frac{d}{n}|A||B|\Bigg|\le\lambda \sqrt{|A||B|\left(1-\frac{|A|}{n}\right)\left(1-\frac{|B|}{n}\right)}.$$}
\end{lemma}

The next lemma bounds the number of subtrees of a given order in a $d$-regular graph:
\begin{lemma} \label{trees}
\textit{Let $t_k(G)$ denote the number of $k$-vertex trees contained in $d$-regular graph $G$ on $n$ vertices. Then,
$$t_k(G)\le n\frac{k^{k-2}d^{k-1}}{k!}.$$}
\end{lemma}
This follows directly from Lemma 2.1 of \cite{trees}.

Let $D_{k}(G):=\left\{v\in V(G):d_G(v)>k\right\}$. We will use the following lemma, which bounds the typical number of edges incident to high degree vertices in $G_p$:
\begin{lemma} \label{high-degree}
\textit{Let $G=(V,E)$ be a $d$-regular graph on $n$ vertices. Let $\epsilon>0$ be a small enough constant and let $p\le \frac{2}{d}$. Form $G_p$ by retaining each edge of $G$ with probability $p$. Let $X$ be the random variable
\begin{align*}
    X=\sum_{v\in D_{5\ln\left(\frac{1}{\epsilon}\right)}(G_p)}d_{G_p}(v).
\end{align*}
Then, \textbf{whp}, 
$$X\le n\epsilon^{\ln\ln\left(\frac{1}{\epsilon}\right)}.$$}
\end{lemma}
\begin{proof}
Let $k=n\epsilon^{\ln\ln\left(\frac{1}{\epsilon}\right)}$. Assume that $X>k$. Then, there exists a subset of $V$, denote it by $S$, formed by vertices of degree at least $5\ln\left(\frac{1}{\epsilon}\right)$ such that 
\begin{align*}
    e_{G_p}(S)+e_{G_p}(S, S^C)\ge k,
\end{align*}
and 
\begin{align*}
    |S|\le \frac{2k}{5\ln\left(\frac{1}{\epsilon}\right)}+1\le\frac{k}{2\ln\left(\frac{1}{\epsilon}\right)}.
\end{align*}
We have at most $\binom{n}{|S|}\le \binom{n}{k/2\ln\left(\frac{1}{\epsilon}\right)}$ ways to choose a set of size $|S|$. We then have at most $\binom{|S|d}{k}\le \binom{kd/2\ln\left(\frac{1}{\epsilon}\right)}{k}$ ways to choose $k$ of the edges touching $S$, and include them in $G_p$ with probability $p^k$. Therefore, the probability that $X>k$ is at most:
\begin{align*}
    \binom{n}{\frac{k}{2\ln\left(\frac{1}{\epsilon}\right)}}\binom{\frac{kd}{2\ln\left(\frac{1}{\epsilon}\right)}}{k}p^k&
    \le \left(\left(\frac{2en\ln\left(\frac{1}{\epsilon}\right)}{n\epsilon^{\ln\ln\left(\frac{1}{\epsilon}\right)}}\right)^{\frac{1}{2\ln\left(\frac{1}{\epsilon}\right)}}\frac{ekd}{2k\ln\left(\frac{1}{\epsilon}\right)}\cdot\frac{2}{d}\right)^k\\
    & \le
    \left(\left(\frac{6\ln\left(\frac{1}{\epsilon}\right)}{\epsilon^{\ln\ln\left(\frac{1}{\epsilon}\right)}}\right)^{\frac{1}{2\ln\left(\frac{1}{\epsilon}\right)}}\frac{e}{\ln\left(\frac{1}{\epsilon}\right)}\right)^k\\
    &= \left(\frac{\left(6\ln\left(\frac{1}{\epsilon}\right)\right)^{\frac{1}{2\ln\left(\frac{1}{\epsilon}\right)}}}{\ln^{1/2}\left(\frac{1}{\epsilon}\right)}\cdot \frac{e}{\ln\left(\frac{1}{\epsilon}\right)}\right)^k\\
    &\le \left(\frac{e}{\ln\left(\frac{1}{\epsilon}\right)}\right)^k=o(1).
\end{align*}
\end{proof}

The next lemma bounds the typical number of edges incident to connected subsets:
\begin{lemma} \label{incident-edges}
\textit{Let $G=(V,E)$ be a $d$-regular graph. Let $p\le\frac{2}{d}$. Form $G_p$ by retaining each edge of $G$ with probability $p$. Then, \textbf{whp} for all $S\subseteq V(G)$ such that $|S|=k\ge \ln n$ and $G_p[S]$ is connected,
$$e_{G_p}(S)+e_{G_p}(S,S^C)<10k.$$
}
\end{lemma}
\begin{proof}
Since any connected set in $G_p$ has a spanning tree, it is sufficient to show the statement is true for the edges incident to trees of order $k$ in $G_p$. By Lemma \ref{trees}, there are at most $n(ed)^{k-1}$ trees on $k$ vertices in $G$. The probability such a tree falls into $G_p$ is $p^{k-1}$. A set of $k$ vertices has at most $kd$ edges touching it in $G$. We have at most $\binom{kd}{9k}$ ways to choose the additional $9k$ edges incident to the set, and they are in $G_p$ with probability $p^{9k}$. As such, by the union bound, the probability of an event violating the statement of the lemma is at most:
\begin{align*}
    n(ed)^{k-1}p^{k-1}\binom{kd}{9k}p^{9k}&\le n\cdot \left(2e\right)^{k-1}\left(\frac{2e}{9}\right)^{9k}\\
    &\le n\exp(-2k)=o(1/n),
\end{align*}
for $k\ge \ln n$. The union bound over the $<n$ possible values of $k$ completes the proof.
\end{proof}

The final lemma of this section will be a key tool in proving Theorem \ref{edge-expansion}. In order to prove it, we will make use of the Breadth First Search (BFS) algorithm. 

This variation of the BFS constructs $G_p$ while exploring the spanning trees of the connected components. The algorithm maintains three sets of vertices: $S$, the vertices whose exploration is complete; $Q$, the vertices currently being explored, kept in a queue; and $T$, the vertices that have not been explored yet. The algorithm starts with $S=Q=\varnothing$ and $T=V(G)$, and ends when $Q\cup T=\varnothing$. At each step, if $Q$ is non-empty, the algorithm queries $T$ for neighbours of the first vertex in $Q$, according to the edges of $G$ and an order $\sigma$ on $V$. The algorithm is fed $X_i$, $0\le i\le \frac{nd}{2}$, i.i.d Bernoulli$(p)$ random variables, each corresponding to a positive (with probability $p$) or negative (with probability $1-p$) answer to such a query. Thus, each $X_i$ corresponds to some edge of $G$ (noting that this edge depends on the execution of the algorithm), which is included in $G_p$ if $X_i=1$ and is excluded otherwise. If $Q$ is non-empty and we discovered a neighbour of $v\in Q$ in $T$, we place the neighbour as the last vertex in $Q$ and continue. If $Q$ is non-empty and the first vertex in $Q$ has no more queries to ask, then we move the first vertex of $Q$ to S. If $Q=\varnothing$, we move the next vertex from $T$ (according to $\sigma$) into $Q.$  Observe that the BFS uncovers a spanning forest of $G_p$. Formally, in order to obtain a graph distributed according to $G_p$, after the run of the BFS we query all edges in $G$ left unqueried and include them in $G_p$ according to the remaining random bits $X_i$. However, this final step in the run of the algorithm will not be of any importance to us, and henceforth when referring to the BFS we consider only the part of the run where we uncover a spanning forest of $G_p$.

We will make use of the following properties of the BFS:
\begin{enumerate}
        \item A vertex moves from $Q$ to $S$ only if it has no more neighbours in $T$. Thus, at any moment $t$ all the edges between $S$ and $T$ have been queried;
        \item While $Q$ is non-empty, we are exploring the same connected component; 
        \item At any moment $t$, we have that $|S\cup Q|\ge 1+\sum_{i=1}^tX_i$, since the first vertex moves into $Q$ without a query, and every positive answer to a query corresponds to a vertex that moves from $T$ into $Q$ (and perhaps later on into $S$); and,
        \item At any moment $t$, we have asked no queries between the vertices of $Q$ and $T$, except perhaps for the first vertex of $Q$.
\end{enumerate}

We note that the initial analysis of the BFS will be similar to that of the Depth First Search (DFS) in \cite{k and s}.

\begin{lemma} \label{BFS}
\textit{Let $\epsilon>0$ be a small enough constant and let $\delta\le \epsilon^4$. Let $G=(V,E)$ be an $(n,d,\lambda)$-graph, with $\frac{\lambda}{d}\le \delta$. Let $p=\frac{1+\epsilon}{d}$. Form $G_p$ by retaining each edge of $G$ independently with probability $p$. Then, the probability there is a connected component $B$, $|B|=k$, in $G_p$ with $\frac{\epsilon^2n}{16}\le k\le \frac{11 \epsilon n}{10}$ is at most $\exp\left(-\frac{\epsilon^2k}{203^2}\right).$}
\end{lemma}
\begin{proof}
We require the following claim, whose proof we defer to the end of this proof:
\begin{claim} \label{large-Q}
\textit{With probability at least $1-\exp\left(-\frac{\epsilon^3n}{201^2}\right)$, there is an interval starting at least as early as $\frac{\epsilon^2nd}{20}$, where the BFS uncovers a connected component of size larger than $\frac{11\epsilon n}{10}$.}
\end{claim}

We now consider the three possible cases: we finished discovering $B$ prior to $\frac{\epsilon^2nd}{20}$; $B$ is the connected component asserted by Claim \ref{large-Q}; we began discovering $B$ subsequent to discovery of the connected component described in Claim \ref{large-Q}. 

If we finished discovering $B$ prior to $\frac{\epsilon^2nd}{20}$, then by properties of the BFS we have that $|B|\le 1+\sum_{i=1}^{\frac{\epsilon^2nd}{20}}X_i$. As such, by a typical Chernoff-type bound (see Appendix A in \cite{probablistic method}), we obtain that with probability at least \begin{align*}
    1-\exp\left(-\frac{\frac{\epsilon^4n^2}{16^2}}{\frac{3\cdot\epsilon^2n}{20}}\right)\ge 1-\exp\left(-\frac{\epsilon^2n}{40}\right),    
\end{align*}
we have that $|B|< \frac{\epsilon^2n}{15}.$

If $B$ is the connected component described in Claim \ref{large-Q}, then with probability at least $1-\exp\left(-\frac{\epsilon^3n}{201^2}\right)$, we have that $|B|>\frac{11\epsilon n}{10}$.

Suppose towards contradiction that we discovered some connected component $B$ of size at least $\frac{\epsilon^2n}{15}$, and that we began discovering $B$ subsequent to the discovery of the connected component described in Claim \ref{large-Q}. By Claim \ref{large-Q}, with probability at least $1-\exp\left(-\frac{\epsilon^3n}{202^2}\right)$ we have discovered a component of size at least $\frac{11\epsilon n}{10}$, which we denote by $M$. Note that by properties of the BFS, by the time we have finished discovering $M$ all the edges of $G$ between $M$ and $T$ have been queried (with all queries answered in the negative). Therefore, we could have no queries between $B$ and $M$, and hence we have an interval in the BFS whose length is at most $|B|d-e_G(B,M)$, where we received at least $|B|-1$ positive answers. By the above bound on $|M|$ and by the expander mixing lemma (Lemma \ref{eml}), with probability at least $1-\exp\left(-\frac{\epsilon^3n}{202^2}\right)$ we have that:
\begin{align*}
    e_G(B,M)&\ge \frac{d}{n}|B||M|-\lambda\sqrt{|B||M|}\\
    &\ge \frac{11\epsilon |B|d}{10}-\lambda\sqrt{\frac{11\epsilon n|B|}{10}}\\
    &\ge \frac{12\epsilon |B|d}{11},
\end{align*}
where the last inequality is due to $|B|\ge \frac{\epsilon^2n}{15}$ and $\frac{\lambda}{d}\le \delta \le \epsilon^4$. Therefore, with probability at least $1-\exp\left(-\frac{\epsilon^3n}{202^2}\right)$, we have an interval whose length is at most $$|B|d-e_G(B,M)\le \left(1-\frac{12\epsilon}{11}\right)|B|d,$$
where we received at least $|B|-1$ positive answers. Considering the $n$ possible starts of the interval together with a Chernoff-type bound, this happens with probability at most:
\begin{align*}
    n\cdot P\left[Bin\left(\left(1-\frac{12\epsilon}{11}\right)|B|d, p\right)\ge |B|-1\right]\le \exp\left(-\frac{\epsilon^2|B|}{50}\right),
\end{align*}
and thus the probability of an event violating the statement of the Lemma is at most $$\exp\left(-\frac{\epsilon^2|B|}{50}\right)+\exp\left(-\frac{\epsilon^3n}{202^2}\right)\le \exp\left(-\frac{\epsilon^2|B|}{203^2}\right),$$
completing the proof.

\begin{proof}[Proof of Claim \ref{large-Q}]
We begin by showing that with probability at least $1-\exp\left(-\frac{\epsilon^3n}{41^2}\right)$, we have that $Q$ does not empty in the interval
\begin{align*}
    I=\left[\frac{\epsilon^2nd}{20}, \frac{3\epsilon nd}{4}\right].
\end{align*}
Assume $Q$ empties at some moment $t\in I.$ Specifically, we let $t$ be the first moment in that interval such that $Q$ is empty. 

Note, we may assume that at $t$ we have $|S|\le \frac{n}{3}$. Otherwise, as the vertices move between the sets one by one, there is some moment $t'\le t\le \frac{3\epsilon nd}{4}$ where $|S|=\frac{n}{3}$. By the expander mixing lemma, we then had by the moment $t'$ at least
\begin{align*}
    e_G(S,S^C)&\ge \frac{d}{n}\frac{2n}{3}\frac{n}{3}-\lambda\sqrt{\frac{2n^2}{9}}\\
    &\ge \frac{nd}{9}>\frac{3\epsilon nd}{4}
\end{align*}
queries, a contradiction.

Furthermore, at $t$, we have by properties of the BFS that $|S|\ge \sum_{i=1}^t X_i$. By a Chernoff-type bound, we obtain that with probability at least $1-\exp\left(-\frac{\epsilon t}{40^2d}\right)\ge 1-\exp\left(-\frac{\epsilon^3n}{41^2}\right)$, we have $|S|\ge \frac{(1+\epsilon)t}{d}-\frac{\epsilon t}{39d}.$ As such, by the expander mixing lemma, with the same probability
\begin{align*}
    e_G(S,S^C)&\ge \frac{d}{n}(n-|S|)|S|-\lambda\sqrt{(n-|S|)|S|}\\
    &\ge \left(1-2\delta-\frac{\left(1+\epsilon-\frac{\epsilon}{39}\right)t}{dn}\right)\left(1+\epsilon-\frac{\epsilon}{39}\right)t,
\end{align*}
where the second inequality uses the fact that since we assumed that $|S|\le \frac{n}{3}$, the expression is minimised by the smallest possible value of $|S|$. We note that this accounts for the queries between $S$ and $T$ by $t$. 
For any $t\le \frac{3\epsilon nd}{4}$, we then have that with the same probability:
\begin{align*}
    t&\ge \left(1-2\delta-\frac{\left(1+\frac{38\epsilon}{39}\right)t}{dn}\right)\left(1+\frac{38\epsilon}{39}\right)t\\
    &\ge \left(1-\frac{4\epsilon}{5}\right)\left(1+\frac{38\epsilon}{39}\right)t>t,
\end{align*}
a contradiction.

We now continue the analysis of the algorithm, albeit in a slightly different manner. By the properties of the BFS, at the moment $t=\frac{3\epsilon nd}{4}$, there have been no queries between the vertices of $Q$ and $T$ (except perhaps for the first vertex $Q$).
We set $Q_0$, $T_0$ and $S_0$ to be the sets $Q$ without its first vertex, $T$ and $S$, respectively, at the moment $\frac{3\epsilon nd}{4}$. For any $i\ge 1$, as long as $Q_{i-1}$ is nonempty, we define $Q_i$ to be the neighbourhood of $Q_{i-1}$ in $T$, $S_i=S_{i-1}\cup Q_{i-1}$ and $T_i=T_{i-1}\setminus Q_{i}$. Furthermore, we call each such step an \textit{iteration}, and $Q_i, S_i$ and $T_i$ are the relevant sets at iteration $i$.
We claim that, with probability at least $1-\exp\left(-\frac{\epsilon^3n}{50^2}\right)$, there is an iteration where we uncover a connected component whose size is larger than $\frac{11\epsilon n}{10}$, completing the proof. 

We begin by considering the set $Q_0$. Assume that $|Q_0|\le \frac{\epsilon^2n}{9}$. Then, as before, by the properties of the BFS and a Chernoff-type bound, we have with probability at least $1-\exp\left(-\frac{\epsilon^3n}{41^2}\right)$ that: 
\begin{align}
|S_0|\ge \frac{3(1+\epsilon)\epsilon n}{4}-\frac{3\epsilon^2n}{130}-\frac{\epsilon^2n}{9}\ge \frac{3\epsilon n}{4}+\frac{3\epsilon^2n}{5}.
\end{align}
By the same arguments as before, we may further assume that $|S_0|\le \frac{n}{3}$. Therefore (with the same probability), by the expander mixing lemma
\begin{align*}
    e_G(S_0,T_0)&\ge \frac{d}{n}|S_0||T_0|-\lambda\sqrt{|S_0||T_0|}\\
            &\ge (1-5\delta)\left(\frac{3\epsilon}{4}+\frac{3\epsilon^2}{5}\right)\left(1-\frac{3\epsilon}{4}-\epsilon^2\right)nd\\
            &\ge \frac{3\epsilon nd}{4}+\frac{\epsilon^2 nd}{80}>\frac{3\epsilon nd}{4},
\end{align*}
where in the second inequality we used the fact that since we assume that $|S_0|\le \frac{n}{3}$, the expression is minimised by the smallest possible value of $|S_0|$. Thus, with probability at least $1-\exp\left(-\frac{\epsilon^3n}{41^2}\right)$, we have that $|Q_0|\ge\frac{\epsilon^2n}{9}$.

We also require the following claim:
\begin{claim} \label{eml-cor}
\textit{Let $\alpha, \beta>0$ and $A\subseteq V$ be such that $|A|\ge \beta n$. Define:
\begin{align*}
    B=\left\{v\in V\setminus A:d_G(v,A)\le (1-\alpha)\frac{|A|d}{n}\right\}.
\end{align*}
Then $|B|\le \frac{\delta^2n}{\alpha^2\beta}$.
}
\end{claim}
\begin{proof}
By the definition of $B$, we have that $e_G(A,B)\le (1-\alpha)\frac{|A||B|d}{n}$. On the other hand, by the expander mixing lemma, we have that:
\begin{align*}
    e_G(A,B)\ge \frac{|A||B|d}{n}-\lambda\sqrt{|A||B|}.
\end{align*}
Combining these together, we obtain:
\begin{align*}
    \frac{|A||B|d}{n}-\lambda\sqrt{|A||B|}\le (1-\alpha)\frac{|A||B|d}{n}.
\end{align*}
Since $|A|\ge \beta n$, we thus have that $|B|\le \frac{\delta^2n}{\alpha^2\beta}$, as required.
\end{proof}
Now, using Claim \ref{eml-cor} with $A=Q_i$ and $\alpha=\epsilon^2$, we have that the number of vertices which have less than $\frac{(1-\epsilon^2)|Q_i|d}{n}$ neighbours in $Q_i$ is at most $\frac{n}{|Q_i|}\cdot\frac{\delta^2n}{\epsilon^4}\le \frac{\epsilon^4n^2}{|Q_i|}$. Denote the set of these vertices by $B_i$. As such, every vertex in $T_i\setminus B_i$ has probability at least $1-(1-p)^\frac{(1-\epsilon^2)|Q_i|d}{n}\ge \frac{(1+\epsilon-2\epsilon^2)|Q_i|}{n}$ to be in $Q_{i+1}$, and these events are independent for each vertex. Hence, at iteration $i+1$, the number of vertices in $Q_{i+1}$ stochastically dominates 
\begin{align*}
    Bin\left(|T_i|-|B_i|, \frac{(1+\epsilon-2\epsilon^2)|Q_i|}{n}\right).
\end{align*}
Since we prefer to express the above in terms of $Q_i$, observe that:
\begin{align*}
    |T_i|-|B_i|\ge n-|S_0|-\sum_{j=1}^i|Q_j|-\frac{\epsilon^4n^2}{|Q_i|}.
\end{align*}
Recall that with probability at least $1-\exp\left(-\frac{\epsilon^3n}{41^2}\right)$, we have that $Q$ does not empty in the interval $I$. With probability at least $1-\exp\left(-\frac{\epsilon^3n}{41^2}\right)$, we have that $\sum_{j=1}^{\frac{3\epsilon nd}{4}}X_j\le \frac{31\epsilon n}{40}$. Furthermore, consider the sets $S$ and $T$ at the last moment before $\frac{\epsilon^2nd}{20}+1$ where $Q$ was empty. If we had that $|S|\ge \epsilon^2n$, then by the expander mixing lemma at that moment
\begin{align*}
    e_G(S, T)\ge \epsilon^2nd-2\lambda\epsilon n>\frac{\epsilon^2nd}{20},
\end{align*}
a contradiction. By properties of the BFS, we may thus conclude that with probability at least $1-\exp\left(-\frac{\epsilon^3n}{42^2}\right)$, we have that $|S_0|\le \frac{31\epsilon n}{40}+\epsilon^2n$. Finally, for the simplicity of computation, we stop the process before $|Q_i|\le \frac{7\epsilon^2n}{90}$, and hence we have that $\frac{\epsilon^4n^2}{|Q_i|}\le 15\epsilon^2n$. Therefore, with probability at least $1-\exp\left(-\frac{\epsilon^3n}{45^2}\right)$, we have that:
\begin{align*}
    |T_i|-|B_i|&\ge n-|S_0|-\sum_{j=1}^i|Q_j|-\frac{\epsilon^4n^2}{|Q_i|}\\
    &\ge \left(1-\frac{31\epsilon}{40}-20\epsilon^2-\frac{\sum_{j=1}^i|Q_j|}{n}\right)n.
\end{align*}

Let
\begin{align*}
    N_i=\bigcup_{j=1}^{i}Q_j, \ q_i=\frac{|Q_i|}{n}, \text{ and } n_i=\frac{|N_i|}{n}.
\end{align*}
By a typical Chernoff-type bound, we have with probability at least
\begin{align*}
    1-\exp\left(-\frac{\epsilon|Q_i|}{46^2}\right)\ge 1-\exp\left(-\frac{\epsilon^3n}{200^2}\right),
\end{align*}
that
\begin{align}
    q_{i+1}\ge \left(1+\frac{44\epsilon}{45}\right)\left(1-\frac{31\epsilon}{40}-\sum_{j=1}^iq_j-20\epsilon^2\right)q_i\ge \left(1+\frac{\epsilon}{5}-(1+\epsilon)\sum_{j=1}^iq_j\right)q_i.
\end{align}

We continue by assuming that $(3)$ holds. Denote by $i_0$ the first iteration where $n_{i_0}\ge\frac{\frac{\epsilon}{5}}{1+\epsilon}$, while noting that for any $i<i_0$, we have by $(3)$ that $q_{i+1}\ge q_i\ge\frac{\epsilon^2}{9}$. Let $i_1$ be the last iteration before $q_i<\frac{7\epsilon^2}{90}$. Assume that $n_{i_1}<\frac{7\epsilon}{20}$. Then, for any $i_0<i< i_1$, we have by $(3)$ that
\begin{align*}
    q_{i+1}&\ge \left(1+\frac{\epsilon}{5}-(1+\epsilon)\frac{7\epsilon}{20}\right)q_i\\
    &>\left(1-\frac{\epsilon}{5}\right)q_i.
\end{align*}
Thus,
\begin{align*}
    \sum_{i=i_0+1}^{i_1}q_i&\ge\frac{\frac{\epsilon^2}{9}-q_{i_1}}{\frac{\epsilon}{5}}\\
    &\ge \frac{5\epsilon}{9}-\frac{7\epsilon}{18(1-\frac{\epsilon}{5})}\\ &>\frac{99\epsilon}{600}.
\end{align*}
However, we then have that
\begin{align*}
    n_{i_1}\ge \frac{\frac{\epsilon}{5}}{1+\epsilon}+\frac{99\epsilon}{600}>\frac{7\epsilon}{20}.
\end{align*}
Hence, we may conclude that with probability at least $1-\exp\left(-\frac{\epsilon^3n}{200^2}\right)$, $|N_{i_1}|=n\cdot n_{i_1}\ge \frac{7\epsilon n}{20}$, and thus with probability at least $1-\exp\left(-\frac{\epsilon^3n}{201^2}\right)$, we have discovered a component of size at least
\begin{align*}
    \frac{7\epsilon n}{20}+\frac{3\epsilon n}{4}=\frac{11\epsilon n}{10},
\end{align*}
where we used both the above bound on $|N_{i_1}|$ and our lower bound on $|S_0|$ as given in $(2)$.
\end{proof}
\end{proof}

\section{Expansion and Expanders}
In this section, we will prove Theorems \ref{vertex-expansion}, \ref{edge-expansion}, and \ref{expanding-subgraph}. The proofs of Theorems \ref{vertex-expansion} and \ref{edge-expansion} are somewhat interwoven.  

We begin with the proof of the first property of Theorem \ref{vertex-expansion}, which also implies the same asymptotic edge-expansion factor for connected subsets. Utilising Lemma \ref{BFS} we then show that, in fact, the same edge-expansion factor holds \textbf{whp} for connected subsets of size up to $\frac{12\epsilon n}{11}$ (Lemma \ref{edge-expansion-connected-sets}). Having the asymptotic bound on the edge-expansion of connected sets of any size between $\frac{16\ln n}{\epsilon^2}$ and $\frac{12\epsilon n}{11}$, we will be able to derive Theorem \ref{edge-expansion}. This, together with an asymptotic bound on the sum of degrees of high-degree vertices (Lemma \ref{high-degree}), will imply the second property of Theorem \ref{vertex-expansion}. We conclude the section with the proof of Theorem \ref{expanding-subgraph}.

For the proof of the first property of Theorem \ref{vertex-expansion}, we will directly bound the probability of the event violating the statement of the theorem.
\begin{proof} [Proof of Property 1 of Theorem \ref{vertex-expansion}]
Consider the event $\mathcal{A}_k$,
\begin{align*}
    \mathcal{A}_k=\left\{\exists S\subseteq V(G), |S|=k, G_p[S]\text{ is connected} \ \& \  \big|N_{G_p}(S)\big|<\frac{\epsilon^2k}{40\ln\left(\frac{1}{\epsilon}\right)}\right\}.
\end{align*}
We will show that for $\frac{16\ln n}{\epsilon^2}\le k\le \frac{\epsilon^2n}{50}$, $P\left[\mathcal{A}_k\right]=o\left(\frac{1}{n}\right)$, and therefore by the union bound over the $<n$ possible values of $k$, the probability of an event violating the statement of the theorem is $o(1)$.

Let $S$, $|S|=k$, be a connected set in $G_p$. Since it is connected, it must have a spanning tree. By Lemma \ref{trees}, we have $n(ed)^{k-1}$ ways to choose a tree of size $k$, and we include its edges in $G_p$ with probability $p^{k-1}$. Now, consider the auxiliary random bipartite graph $\Gamma(S,p)$, whose one side is $S$, the other side is $N_G(S)$, and we include every edge of $G$ between $S$ and $N_G(S)$ in $\Gamma(S,p)$ with probability $p$. Then, we have that $|N_{G_p}(S)|\ge \nu\left(\Gamma(S,p)\right)$, where $\nu(H)$ is the matching number of $H$. Thus, it suffices to bound the probability that a maximum matching in $\Gamma(s,p)$ is smaller than $\frac{\epsilon^2k}{40\ln\left(\frac{1}{\epsilon}\right)}.$

Let us first bound the probability that $\nu\left(\Gamma(S,p)\right)=i$. This is at most the probability that $\Gamma(s,p)$ has a maximal by inclusion matching of size $i$. We have at most $\binom{e_G\left(S,N_G(S)\right)}{i}\le \binom{kd}{i}$ ways to choose a matching $M$ of size $i$. We then need to include the edges of the matching, which happens with probability $p^i$. Due to the maximality of $M$, every edge of $G$ between $S$ and $N_G(S)$ disjoint from $M$ is not in $\Gamma(S,p)$. Thus, we have at least $e_G(S, S^C)-2id$ edges that do not fall into $\Gamma(S,p)$. As in our previous arguments, by the expander mixing lemma:
\begin{align*}
   e_G\left(S, S^C\right)\ge \left(1-\frac{\epsilon^2}{45}\right)dk,
\end{align*}
using our assumptions on $\delta$ and $k$. Thus,
\begin{align*}
    P[\nu\left(\Gamma(S,p)\right)=i]\le \binom{kd}{i}p^i(1-p)^{\left(1-\frac{\epsilon^2}{45}\right)kd-2id}.
\end{align*}
Therefore, we obtain that
\begin{align*}
    P\left[\mathcal{A}_k\right]&\le n(ed)^{k-1}p^k\sum_{i=0}^{\frac{\epsilon^2k}{40\ln\left(\frac{1}{\epsilon}\right)}}\binom{kd}{i}p^i(1-p)^{\left(1-\frac{\epsilon^2}{45}\right)kd-2id}\\
    &\le n\left((1+\epsilon)\exp\left(1-(1+\epsilon)\left(1-\frac{\epsilon^2}{45}\right)\right)\right)^k\sum_{i=0}^{\frac{\epsilon^2k}{40\ln\left(\frac{1}{\epsilon}\right)}}\binom{kd}{i}p^i(1-p)^{-2id}\\
    &\le n\left((1+\epsilon)\exp\left(1-(1+\epsilon)\left(1-\frac{\epsilon^2}{45}\right)\right)\right)^k\left(1+\sum_{i=1}^{\frac{\epsilon^2k}{40\ln\left(\frac{1}{\epsilon}\right)}}\left(\frac{e^4k}{i}\right)^i\right).
\end{align*}
Observe that the ratio of consecutive terms in the above sum is:
\begin{align*}
    \frac{\left(\frac{e^4k}{i}\right)^i}{\left(\frac{e^4k}{i+1}\right)^{i+1}}= \frac{(i+1)^{i+1}}{e^4ki^i}\le \frac{1}{2},
\end{align*}
and thus the sum is at most twice the final term. Hence,
\begin{align*}
    P\left[\mathcal{A}_k\right]&\le n\left((1+\epsilon)\exp\left(1-(1+\epsilon)\left(1-\frac{\epsilon^2}{45}\right)\right)\right)^k\left(1+2\left(\frac{e^440\ln\left(\frac{1}{\epsilon}\right)}{\epsilon^2}\right)^{\frac{\epsilon^2k}{40\ln\left(\frac{1}{\epsilon}\right)}}\right)\\
    &\le 3n\left((1+\epsilon)\exp\left(-\epsilon+\frac{\epsilon^2}{40}+\frac{\epsilon^2}{40\ln\left(\frac{1}{\epsilon}\right)}\ln\left(\frac{41e^4\ln\left(\frac{1}{\epsilon}\right)}{\epsilon^2}\right)\right)\right)^k\\
    &\le 3n\left((1+\epsilon)\exp\left(-\epsilon+\frac{\epsilon^2}{4}\right)\right)^k\\
    &\le 3n\exp\left(-\frac{\epsilon^2k}{8}\right)=o\left(\frac{1}{n}\right),
\end{align*}
where the last inequality is due to $1+x\le \exp\left(x-\frac{3x^2}{8}\right)$ for $x$ small enough, and the equality is since we assumed $k\ge \frac{16\ln n}{\epsilon^2}$. Union bound over the $<n$ different values of $k$ completes the proof.
\end{proof}

We now turn to show that for edge-expansion, the same expansion factor holds \textbf{whp} for \textit{connected} sets up to size $\frac{12\epsilon n}{11}$. For subsets larger than $\frac{\epsilon^2n}{50}$, our arguments are of the same flavour as those used for separators in \cite{high girth} and utilise Lemma \ref{BFS}:
\begin{lemma}\label{edge-expansion-connected-sets}
\textit{Let $\epsilon>0$ be a small enough constant and let $\delta>0$ be such that $\delta\le \epsilon^4$. Let $G=(V,E)$ be an $(n,d,\lambda)$-graph with $\frac{\lambda}{d}\le \delta$. Let $p=\frac{1+\epsilon}{d}$. Form $G_p$ by retaining every edge of $G$ independently with probability $p$. Let $L_1$ be the giant component of $G_p$. Then, there exists an absolute constant $c>0$ such that \textbf{whp} for any $S\subseteq L_1$ such that $\frac{16\ln n}{\epsilon^2}\le |S| \le \frac{12\epsilon n}{11}$ and $G_p[S]$ is connected, we have
\begin{align*}
    \bigg|\partial_{G_p}(S)\bigg|\ge\frac{c\epsilon^2|S|}{\ln\left(\frac{1}{\epsilon}\right)}.
\end{align*}
}
\end{lemma}
\begin{proof}
For $\frac{16\ln n}{\epsilon^2}\le |S|\le \frac{\epsilon^2n}{50}$, this follows immediately from the first property of Theorem 1. We now assume that $\frac{\epsilon^2n}{50}\le |S|\le \frac{12\epsilon n}{11}$. Let $p_1=\frac{1+\epsilon-\epsilon^2/50}{d}$, and let $\rho=\frac{p_1}{p}$. Let $G_{p_1}$ be the subgraph obtained by retaining each edge of $G$ with probability $p_1$. Observe that $G_{p_1}$ can also be obtained by drawing a random graph $G_p$, and then by retaining each edge with probability $\rho$. 

Let $\mathcal{A}_k$ be the following event addressing $G_p$:
\begin{align*}
    \mathcal{A}_k=\Bigg\{\exists S\subseteq V(G), |S|=k, S\text{ is connected in } G_p \ \& \ \bigg|\partial_{G_p}(S)\bigg|<\frac{\epsilon^2k}{400^2\ln\left(\frac{1}{\epsilon}\right)}\Bigg\}.
\end{align*}
Let $\mathcal{B}_k$ be the following event addressing $G_{p_1}$:
\begin{align*}
    \mathcal{B}_k=\left\{\exists S, |S|=k \ \& \ S\text{ is a connected component in } G_{p_1}\right\}.
\end{align*}
Note that $\rho=1-\frac{\epsilon^2}{50}+O(\epsilon^3)$. Thus, going from $G_p$ to $G_{p_1}$ by retaining each edge of $G_p$ with probability $\rho$, given $\mathcal{A}_k$, the probability that $S$ is separated from $L_1$ and remains connected is at least:
\begin{align*}
    (1-\rho)^{\frac{\epsilon^2k}{400^2\ln\left(\frac{1}{\epsilon}\right)}}\rho^{k}&\ge \exp\left(-\left(\frac{\epsilon^2}{40}+\ln\left(\frac{1}{40\epsilon^2}\right)\frac{\epsilon^2}{400^2\ln\left(\frac{1}{\epsilon}\right)}\right)k\right)\\
    &\ge \exp\left(-\frac{\epsilon^2k}{300^2}\right).
\end{align*}
Hence,
\begin{align*}
    P\left[\mathcal{B}_k|\mathcal{A}_k\right]\ge \exp\left(-\frac{\epsilon^2k}{300^2}\right).
\end{align*}
Since $p_1$ can be rewritten as $p_1=\frac{1+\epsilon'}{d}$ with $\epsilon'=\epsilon-\frac{\epsilon^2}{50}$, we have by Lemma \ref{BFS} that for $\frac{\epsilon^2n}{15}\le k \le \frac{3\epsilon n}{2}$, 
\begin{align*}
    P\left[\mathcal{B}_k\right]\le \exp\left(-\frac{\epsilon'^2k}{203^2}\right)\le\exp\left(-\frac{\epsilon^2k}{204^2}\right).
\end{align*}
Thus,
\begin{align*}
    P\left[\mathcal{A}_k\right] \le \exp\left(-\frac{\epsilon^2k}{204^2}+\frac{\epsilon^2k}{300^2}\right)=o\left(\frac{1}{n}\right).
\end{align*}
Union bound over the $<n$ possible values of $k$ completes the proof of the lemma, with $c=\frac{1}{400^2}$.
\end{proof}

Before proving Theorem \ref{edge-expansion}, we require the following lemma as well:
\begin{lemma} \label{volume-lemma}
\textit{Let $\epsilon>0$ be a small enough constant and let $\delta\le \epsilon^4$. Let $G=(V,E)$ be an $(n,d,\lambda)$-graph, with $\frac{\lambda}{d}\le \delta$. Let $p=\frac{1+\epsilon}{d}$. Form $G_p$ by retaining each edge of $G$ independently with probability $p$. Let $c\le \frac{1}{70}$ be a positive constant independent of $\epsilon$. Define the random variable $X$ to be the maximal volume of a family of vertex disjoint subsets $S$, such that $\frac{\ln\left(\frac{1}{\epsilon}\right)}{c\epsilon^2}\le |S| \le \frac{16\ln n}{\epsilon^2}$, $G_p[S]$ is connected and $\big|\partial_{G_p}(S)\big|< \frac{c\epsilon^2}{\ln\left(\frac{1}{\epsilon}\right)}$. Then \textbf{whp} $X\le 2\epsilon^3 n$.}
\end{lemma}
\begin{proof}
We begin with an upper bound on $\mathbb{E}X$. Clearly, evaluating the expected total volume of all trees in $G_p$ of order $\frac{\ln\left(\frac{1}{\epsilon}\right)}{c\epsilon^2}\le k \le \frac{16\ln n}{\epsilon^2}$ whose edge boundary in $G_p$ is at most $\frac{c\epsilon^2k}{\ln\left(\frac{1}{\epsilon}\right)}$, provides an upper bound.

By Lemma \ref{trees}, we have $n(ed)^{k-1}$ ways to choose the tree. We include its edges in $G_p$ with probability $p^{k-1}$. Denote by $m$ the number of edges in the boundary that fall into $G_p$. Then, we have $\binom{kd}{m}$ ways to choose the edges in the boundary that fall into $G_p$, and this happens with probability $p^m$. Observe that 
\begin{align*}
    \binom{kd}{m}p^m\le \left(\frac{ek(1+\epsilon)}{m}\right)^m\le \left(\frac{ek(1+\epsilon)}{\frac{c\epsilon^2k}{\ln\left(\frac{1}{\epsilon}\right)}}\right)^{\frac{c\epsilon^2k}{\ln\left(\frac{1}{\epsilon}\right)}},
\end{align*}
since $(a/x)^x$ is strictly increasing for $0<x<a/e$, and we assumed that $m\le \frac{c\epsilon^2k}{\ln\left(\frac{1}{\epsilon}\right)}$. Furthermore, we have that all but $m$ of the edges in $E_G(S,S^C)$ do not fall into $G_p$. By the expander mixing lemma, we obtain in a manner similar to the previous arguments that:
\begin{align*}
    e_G(S,S^C)-m&\ge \left(1-\frac{k}{n}-2\delta\right)kd-m\\
    &\ge(1-\epsilon^3)kd,
\end{align*}
where we used our assumptions on $k$ and $\delta$ and $m$. Finally, we consider all the $<k$ possible values of $m$. Hence, performing calculations similar to those we employed in proving Property 1 of Theorem \ref{vertex-expansion}, we obtain:
\begin{align*}
    \mathbb{E}X&\le  \sum_{k=\frac{\ln\left(\frac{1}{\epsilon}\right)}{c\epsilon^2}}^{\frac{16\ln n}{\epsilon^2}}k\cdot n(ed)^{k-1}p^{k-1}\cdot k(1-p)^{(1-\epsilon^3)kd}\left(\frac{ek(1+\epsilon)}{\frac{c\epsilon^2k}{\ln\left(\frac{1}{\epsilon}\right)}}\right)^{\frac{c\epsilon^2k}{\ln\left(\frac{1}{\epsilon}\right)}}\\
    &\le n\sum_{k=\frac{\ln\left(\frac{1}{\epsilon}\right)}{c\epsilon^2}}^{\frac{16\ln n}{\epsilon^2}}k^2\exp\left(-\frac{\epsilon^2k}{10}\right)\\
    &\le n\cdot 5\frac{\ln^2\left(\frac{1}{\epsilon}\right)}{c^2\epsilon^4}\exp\left(-\frac{\epsilon^2}{10}\cdot \frac{\ln\left(\frac{1}{\epsilon}\right)}{c\epsilon^2}\right)\\
    & \le n\cdot \frac{5}{c^2}\epsilon^{\frac{1}{10c}-5} \le \epsilon^3n,
\end{align*}
assuming that $c\le \frac{1}{90}$. 

We now proceed to show that $X$ is tightly concentrated. Indeed, consider the edge exposure martingale for $G_p$. Changing one edge in $G$ can add or delete at most two sets to the family in question (recalling that it is a family of vertex disjoint subsets). Therefore, at each step of the exposure, we change the size of $X$ by at most $\frac{32\ln n}{\epsilon^2}$. Thus, by a variant of the Azuma-Hoeffding inequality (see, for example, Theorem 3.9 of \cite{azuma survey}), we have that:
\begin{align*}
    P\left[X\ge2\epsilon^3n\right]&\le P\left[X\ge\mathbb{E}X+\epsilon^3n\right]\\
    &\le \exp\left(-\frac{\epsilon^{6}n^2}{nd\cdot\frac{1+\epsilon}{d}\left(\frac{32\ln n}{\epsilon^2}\right)^2+2\cdot\frac{32\ln n}{\epsilon^2}\cdot\epsilon^3n}\right)=o\left(\frac{1}{n}\right).
\end{align*}
\end{proof}

We are now ready to prove Theorem \ref{edge-expansion}:
\begin{proof}[\textbf{\textit{Proof of Theorem \ref{edge-expansion}}}]
Let $c_{(\ref{edge-expansion-connected-sets})}$ be as in Lemma \ref{edge-expansion-connected-sets}. Consider the connected components of $G_p[S]$. By Lemma \ref{edge-expansion-connected-sets}, \textbf{whp} every component of $G_p[S]$ whose size is $k\ge \frac{16 \ln n}{\epsilon^2}$ contributes at least $\frac{c_{(\ref{edge-expansion-connected-sets})}\epsilon^2k}{\ln\left(\frac{1}{\epsilon}\right)}$ edges to $\partial_{G_p}(S)$, since none of these edges can be in $S$ (otherwise, it would not be a component of $G_p[S]$). Furthermore, since $S\subseteq L_1$ and $L_1$ is connected, every connected component of $G_p[S]$ contributes at least one edge to $\partial_{G_p}(S)$. Therefore, any component of $G_p[S]$ whose size is at most $k\le \frac{\ln\left(\frac{1}{\epsilon}\right)}{c_{(\ref{edge-expansion-connected-sets})}\epsilon^2}$ contributes at least $\frac{c_{(\ref{edge-expansion-connected-sets})}\epsilon^2k}{\ln\left(\frac{1}{\epsilon}\right)}$ edges to $\partial_{G_p}(S)$ as well. We are thus left with connected components of $G_p[S]$, whose size is between  $\frac{\ln\left(\frac{1}{\epsilon}\right)}{c_{(\ref{edge-expansion-connected-sets})}\epsilon^2}$ and $\frac{16\ln n}{\epsilon^2}$.

For any $v\in S$, we define $C_v[S]$ to be the connected component of $G_p[S]$ which includes $v$. Let $B\subseteq S$ be the following set:
\begin{align*}
    B=\left\{v\in S: \big|C_v[S]\big|=k, \frac{\ln\left(\frac{1}{\epsilon}\right)}{c_{(\ref{edge-expansion-connected-sets})}\epsilon^2}\le k \le \frac{16\ln n}{\epsilon^2}, \bigg|\partial(C_v[S])\bigg|< \frac{c_{(\ref{edge-expansion-connected-sets})}\epsilon^2 k}{\ln\left(\frac{1}{\epsilon}\right)}\right\}.
\end{align*}
By Lemma \ref{volume-lemma}, we have that \textbf{whp} $|B|\le 2\epsilon^3n.$

We can thus conclude that \textbf{whp} all but at most $2\epsilon^3n$ of the vertices of $S$ belong to connected components $C_v[S]$ of $G_p[S]$ such that $\bigg|\partial_{G_p}\left(C_v[S]\right)\bigg|\ge \frac{c_{(\ref{edge-expansion-connected-sets})}\epsilon^2\big|C_v[S]\big|}{\ln\left(\frac{1}{\epsilon}\right)}$. Hence, for any $S\subseteq L_1$ such that $\frac{\epsilon^2 n}{50}\le |S| \le \frac{12\epsilon n}{11}$, we have that \textbf{whp}:
\begin{align*}
    \bigg|\partial_{G_p}(S)\bigg|\ge \frac{c_{(\ref{edge-expansion-connected-sets})}\epsilon^2\left(|S|-2\epsilon^3n\right)}{\ln\left(\frac{1}{\epsilon}\right)}\ge \frac{c_{(\ref{edge-expansion-connected-sets})}\epsilon^2|S|}{2\ln\left(\frac{1}{\epsilon}\right)},
\end{align*}
proving the second property of the theorem with $c=c_{(\ref{edge-expansion-connected-sets})}/2$.
\end{proof}

The second part of Theorem \ref{vertex-expansion}, linear sized (not necessarily connected) sets, will build upon Theorem \ref{edge-expansion}:
\begin{proof}[Proof of Theorem \ref{vertex-expansion}: Property 2]
By Theorem \ref{edge-expansion}, for any subset $S\subseteq L_1$ such that $\frac{\epsilon^2n}{50}\le |S| \le \frac{12\epsilon n}{11}$, we have that \textbf{whp}
\begin{align*}
    \bigg|\partial_{G_p}(S)\bigg|\ge \frac{c\epsilon^2|S|}{\ln\left(\frac{1}{\epsilon}\right)}.
\end{align*}
Let $D_{t}\left(N_{G_p}(S)\right):=\left\{v\in N_{G_p}(S):d_{G_p}(v)>t\right\}$. By Lemma \ref{high-degree}, \textbf{whp},
\begin{align*}
    \sum_{v\in D_{5\ln\left(\frac{1}{\epsilon}\right)}\left(N_{G_p}(S)\right)}d_{G_p}(v)\le n\epsilon^{\ln\ln\left(\frac{1}{\epsilon}\right)}.
\end{align*}
Thus, for any subset $S\subseteq L_1$ of relevant size, we have that \textbf{whp}
\begin{align*}
    |N_{G_p}(S)|\ge \frac{\bigg|\partial_{G_p}(S)\bigg|-n\epsilon^{\ln\ln\left(\frac{1}{\epsilon}\right)}}{5\ln\left(\frac{1}{\epsilon}\right)}\ge\frac{c\epsilon^2|S|}{6\ln^2\left(\frac{1}{\epsilon}\right)},
\end{align*}
completing the proof.
\end{proof}

We conclude this section with the proof of Theorem \ref{expanding-subgraph}. The proof uses some ideas present in \cite{expanders} and Theorem 1.4 of \cite{EKK}.
\begin{proof}[Proof of Theorem \ref{expanding-subgraph}]
 Let $M\subseteq L_1$ be a maximal set such that $|M|< \frac{\epsilon n}{6}$ and $|N_{G_p}(M)|< \frac{\epsilon^2|M|}{c\ln^2\left(\frac{1}{\epsilon}\right)}.$ Let $L_1'=L_1\setminus M$. Assume there is some subset $B\subseteq L_1'$ with $|B|\le \frac{|V(L_1')|}{2}$ and $|N_{L_1'}(B)|< \frac{\epsilon^2|B|}{c\ln^2\left(\frac{1}{\epsilon}\right)} .$ Then,
\begin{align*}
    \big|N_{G_p}(M\cup B)\big|< \frac{\epsilon^2|M|}{c\ln^2\left(\frac{1}{\epsilon}\right)}+\frac{\epsilon^2|B|}{c\ln^2\left(\frac{1}{\epsilon}\right)}=\frac{\epsilon^2|M\cup B|}{c\ln^2\left(\frac{1}{\epsilon}\right)}.
\end{align*}
Therefore, due to the maximality of $M$, we obtain $|M\cup B|\ge \frac{\epsilon n}{6}$. However, by Theorem \ref{vertex-expansion}, \textbf{whp} every subset $S\subseteq L_1$ with $\frac{\epsilon^2n}{50}\le |S| \le \frac{12\epsilon n}{11}$ has $|N_{G_p}(S)|\ge \frac{\epsilon^2|S|}{c\ln^2\left(\frac{1}{\epsilon}\right)}$. Hence, $|M\cup B|> \frac{12\epsilon n}{11}$. On the other hand, by our choice of $B$ and $M$, we have that:
\begin{align*}
    |M\cup B|&\le|M|+\frac{|V(L_1)|-|M|}{2}\\
    &=\frac{|V(L_1)|+|M|}{2}\\
    &\le \frac{|V(L_1)|}{2}+\frac{\epsilon n}{12}.
\end{align*}
By Theorem \ref{FKM-thm}, \textbf{whp} $|V(L_1)|\le 2\epsilon n$. Therefore, $|M\cup B|\le \frac{13\epsilon n}{12}<\frac{12\epsilon n}{11},$ a contradiction. 
Thus, \textbf{whp} $L_1'$ has the desired expansion property, and by Theorem \ref{FKM-thm} \textbf{whp} $|L_1'|=|L_1|-|M|\ge \frac{23\epsilon n}{12}-\frac{\epsilon n}{6}=\frac{7\epsilon n}{4}$, concluding the proof.
\end{proof}

\section{The Diameter of the Giant}
Equipped with the expansion properties of connected subsets, we are now able to obtain a bound on the diameter of the largest component:
\begin{proof}[\textit{\textbf{Proof of Theorem \ref{diameter}}}] 
Let $B(v,r)$ denote the ball of radius $r$ around $v$ in $G_p$. By Theorem \ref{FKM-thm}, \textbf{whp} we have that $|L_1|< 2\epsilon n$. Hence, if we can show that for any $v\in V(L_1)$, 
\begin{align*}
    \Bigg|B\left(v,C\frac{\ln\left(\frac{1}{\epsilon}\right)\ln n}{\epsilon^2}\right)\Bigg|\ge \epsilon n,
\end{align*}
for some large enough constant $C>0$, then every two balls of such radius in $L_1$ intersect, and hence the diameter of the giant component of $G_p$ is at most
\begin{align*}
    2C\frac{\ln\left(\frac{1}{\epsilon}\right)\ln n}{\epsilon^2}.
\end{align*}
By Theorem \ref{vertex-expansion}, for $\frac{16\ln n}{\epsilon^2}\le \big|B(v,r)\big| \le \frac{\epsilon^2 n}{50}$, we have \textbf{whp} that $\big|N_{G_p}\left(B(v,r)\right)\big|\ge \frac{c\epsilon^2|B(v,r)|}{\ln\left(\frac{1}{\epsilon}\right)}$. Thus, we have for $r\ge \frac{16\ln n}{\epsilon^2}$ that \textbf{whp}:
\begin{align*}
    \big|B(v, r+1)\big|\ge \min\left\{\frac{\epsilon^2n}{50}, \left(1+\frac{c\epsilon^2}{\ln\left(\frac{1}{\epsilon}\right)}\right)\big|B(v,r)\big|\right\}.
\end{align*}
Hence,
\begin{align*}
    \Bigg|B\left(v,\left(\frac{16}{\epsilon^2}+\frac{1}{\ln\left(1+\frac{c\epsilon^2}{\ln(1/\epsilon)}\right)}\right)\ln n\right)\Bigg|\ge \frac{\epsilon^2n}{50}.
\end{align*}
Set $r_1=\left(\frac{16}{\epsilon^2}+\frac{1}{\ln\left(1+\frac{c\epsilon^2}{\ln(1/\epsilon)}\right)}\right)\ln n$. By Theorem \ref{vertex-expansion}, for any subset $S\subseteq L_1$ with $\frac{\epsilon^2 n}{50}\le |S|\le \frac{12\epsilon n}{11}$, we have that \textbf{whp} $\big|N_{G_p}(S)\big|\ge \frac{c\epsilon^2|B(v,r)|}{\ln^2\left(\frac{1}{\epsilon}\right)}$. Therefore, for any $r\ge r_1$, we have that \textbf{whp}:
\begin{align*}
    \big|B(v,r+1)\big|\ge\min\left\{\frac{12\epsilon n}{11}, \left(1+\frac{c\epsilon^2}{\ln^2\left(\frac{1}{\epsilon}\right)}\right)\big|B(v,r)\big|\right\},
\end{align*}
and thus for some constant $K$ large enough, \textbf{whp}
\begin{align*}
    \Bigg|B\left(v,r_1+K\ln\left(\frac{1}{\epsilon}\right)\right)\Bigg|\ge \frac{12\epsilon n}{11}>\epsilon n.
\end{align*}
Thus, \textbf{whp}, the diameter of $L_1$ is at most
\begin{align*}
    2\left(r_1+K\ln\left(\frac{1}{\epsilon}\right)\right)&=2\left(\frac{16}{\epsilon^2}+\frac{1}{\ln\left(1+\frac{c\epsilon^2}{\ln(1/\epsilon)}\right)}\right)\ln n+2K\ln\left(\frac{1}{\epsilon}\right)\\
    &\le C'\frac{\ln\left(\frac{1}{\epsilon}\right)\ln n}{\epsilon^2},
\end{align*}
for large enough absolute constant $C'$.
\end{proof}

\section{Mixing Time of the Lazy Random Walk}
In this section we prove Theorem 5.

We start with some definitions and brief background (see \cite{LPW} for extensive background on Markov chains and mixing time). Given a graph $G$, the \textit{lazy simple random walk} on $G$ is a Markov chain starting at a vertex $v_0$ chosen according to some distribution $\sigma$, such that for any vertex $v\in V(G)$ the walk stays at $v$ with probability $\frac{1}{2}$, and otherwise moves to a uniformly chosen random neighbour $u$ of $v$. Hence, the transition probability from $v$ to $u$ satisfies $P(v\to u)=\frac{1}{2d(v)}$. For $G$ connected, this Markov chain is irreducible and ergodic and as such has a limit distribution, which we call the \textit{stationary distribution} $\pi$, given by $\pi(v)=\frac{d(u)}{2e(G)}$ for any $v\in V(G)$ (see \cite{LPW}). We are interested in estimating how quickly this Markov chain converges to its limit distribution. For that, recall that the \textit{total variation distance} $d_{TV}$ between two distributions $p_1$ and $p_2$ on $V(G)$ is defined by:
\begin{align*}
    d_{TV}(p_1,p_2):=\max_{A\subset V(G)}\bigg|p_1(A)-p_2(A)\bigg|.
\end{align*}
Let $P^t(v,\cdot)$ denote the distribution on $V(G)$ given by starting the lazy random walk at $v\in V(G)$ and running for $t$ steps. Letting
\begin{align*}
    d(t):=\max_{v\in V(G)}d_{TV}\left(P^t(v,\cdot),\pi\right),
\end{align*}
the \textit{mixing time} of the lazy random walk is then defined by:
\begin{align*}
    t_{mix}:=\min\left\{t:d(t)\le \frac{1}{4}\right\}.
\end{align*}
Now, for any $S\subseteq V(G)$, let
\begin{align*}
    \pi(S):=\sum_{v\in S}\pi(v)=\frac{2e(S)+e(S,S^C)}{2e(G)}.
\end{align*}
We further define:
\begin{align*}
    Q(S):=\sum_{v\in S, u\in S^C}\pi(v)P(v\to u)=\frac{e(S,S^C)}{4e(G)}.
\end{align*}
The \textit{conductance} $\Phi(S)$ of $S$  is then given by:
\begin{align*}
    \Phi(S):=\frac{Q(S)}{\pi(S)\pi(S^C)}=\frac{e(S, S^C)}{2\left(2e(S)+e(S,S^C)\right)\pi(S^C)},
\end{align*}
and we note that since $Q(S)=Q(S^C)$, we have that $\Phi(S)=\Phi(S^C)$. Let $\pi_{\min}=\min_{v\in V(G)}\pi(v)$. For $p>\pi_{\min}$, we define:
\begin{align*}
    \Phi(p):=\min\left\{\Phi(S): S\subseteq V(G), p/2\le \pi(S)\le p, \text{S is connected in } G\right\},
\end{align*}
if there is no such subset $S$, we set $\Phi(p)=1$. The following theorem due to Fountoulakis and Reed \cite{FR1} bounds the mixing time through the conductance of connected sets:
\begin{thm}[Theorem 1 of \cite{FR1}] \label{mixing-time-tool}
\textit{There exists an absolute constant $K$ such that
\begin{align*}
    t_{mix}\le K\sum_{j=1}^{\log_2\pi_{\min}^{-1}}\Phi^{-2}\left(2^{-j}\right).
\end{align*}}
\end{thm}

Throughout the rest of this section, we consider the mixing time of the lazy random walk on the giant component. Formally, let $\epsilon>0$ be a small enough constant and let $\delta>0$ be such that $\delta<\epsilon^4$. Let $G$ be an $(n,d,\lambda)$-graph with $\frac{\lambda}{d}\le \delta$. Let $p=\frac{1+\epsilon}{d}$. Let $G_p$ be the graph obtained by retaining each edge of $G$ independently with probability $p$. By Theorem \ref{FKM-thm}, \textbf{whp} there is a unique giant component in $G_p$ which we denote by $L_1$. Below, $e(S)$ will stand for $e_{G_p}(S)$ and $e\left(S,S^C\right)$ will stand for $\big|\partial_{G_p}(S)\big|$. Equipped with these notation, we now establish a couple of lemmas.

\begin{lemma}\label{mixing-time-lemma1}
\textit{\textbf{Whp} for any $S\subseteq L_1$ such that $S$ is connected in $G_p$ and $e(S)+e(S,S^C)\ge \frac{160\ln n}{\epsilon^2}$, we have that $|S|\ge \frac{16\ln n}{\epsilon ^2}$.}
\end{lemma}
\begin{proof}
We prove the contrapositive. Since $S\subseteq L_1$, by Theorem \ref{FKM-thm} \textbf{whp} $S$ is contained in some connected subset $S'\subseteq L_1$ with $|S'|=\frac{16\ln n}{\epsilon^2}$. Observing that $e(S)+e(S,S^C)$ increases when moving from $S$ to $S'\supseteq S$, we have by Lemma \ref{incident-edges} that \textbf{whp}
\begin{align*}
    e(S)+e(S,S^C)\le e(S')+e(S', S'^C)\le \frac{160\ln n}{\epsilon^2}.
\end{align*}
\end{proof}

Using Lemma \ref{mixing-time-lemma1}, we are now able to bound the conductance of relevant connected sets. We will make use of the notion of \textit{excess} of a graph. Recall that the excess of a connected graph $G=(V,E)$ is defined as $|V|-|E|+1$.
\begin{lemma} \label{mixing-time-lemma2}
\textit{There exist positive constants $c$ and $C$ such that \textbf{whp}, for every $S\subseteq L_1$ with $S$ connected in $G_p$ and $\frac{C\ln n}{\epsilon^3 n}\le \pi(S)\le \frac{1}{2}$,
$$\Phi(S)\ge c\frac{\epsilon^2}{\ln\left(\frac{1}{\epsilon}\right)}.$$}
\end{lemma}
\begin{proof}
It can be easily read from the proof of Theorem 2 in \cite{FKM} that \textbf{whp} the excess of $L_1$ is at most $\epsilon^2n$. Furthermore, $S$ is connected and thus $e(S)\ge |S|-1$. Since $\pi(S)=\frac{2e(S)+e(S,S^C)}{2e(L_1)}\le \frac{1}{2}$, we obtain that \textbf{whp},
\begin{align*}
    |S|\le 1+e(S)\le 1+\frac{e(L_1)-e(S,S^C)}{2}\le 1+\frac{|L_1|+\epsilon^2n}{2}\le \frac{12\epsilon n}{11},
\end{align*}
where the last inequality holds \textbf{whp} by Theorem \ref{FKM-thm}. 

On the other hand, $L_1$ is connected, and hence $e(L_1)\ge |L_1|-1$. Since $\pi(S)\ge \frac{C\ln n}{\epsilon^3 n}$, we obtain that \textbf{whp}
\begin{align*}
    e(S)+e(S,S^C)\ge \frac{2e(S)+e(S,S^C)}{2}\ge \frac{C\ln n\cdot e(L_1)}{\epsilon^3n}\ge \frac{C\ln n}{\epsilon^2},
\end{align*}
where the last inequality once again holds \textbf{whp} by Theorem \ref{FKM-thm} and $e(L_1)\ge |L_1|-1$. Hence, choosing $C=160$, we obtain by Lemma \ref{mixing-time-lemma1} that \textbf{whp} $|S|\ge\frac{16\ln n}{\epsilon ^2}$.

All in all, we have that \textbf{whp} $\frac{16\ln n}{\epsilon^2}\le |S|\le \frac{12\epsilon n}{11}$. Thus, by Theorem \ref{vertex-expansion} and Theorem \ref{edge-expansion}, \textbf{whp} $e(S,S^C)\ge \frac{c'\epsilon^2|S|}{\ln\left(\frac{1}{\epsilon}\right)}$, and by Lemma \ref{incident-edges} \textbf{whp} $2e(S)+e(S,S^C)\le 2\left(e(S)+e(S,S^C)\right)\le 20|S|$. Therefore, \textbf{whp}
\begin{align*}
    \Phi(S)&=\frac{e(S, S^C)}{2\left(2e(S)+e(S,S^C)\right)\pi(S^C)}\\
    &\ge \frac{c'\epsilon^2|S|}{\ln\left(\frac{1}{\epsilon}\right)\cdot 2\cdot 20|S|}.
\end{align*}
Hence we obtain that $\Phi(S)\ge c\frac{\epsilon ^2}{\ln\left(\frac{1}{\epsilon}\right)},$
for a fitting choice of $c.$
\end{proof}

We are now ready to prove Theorem \ref{mixing-time}.
\begin{proof}[Proof of Theorem \ref{mixing-time}]
By Theorem \ref{mixing-time-tool}, there is an absolute constant $K$ such that
$$t_{mix}\le K\sum_{j=1}^{\log_2\pi_{\min}^{-1}}\Phi^{-2}\left(2^{-j}\right).$$
Let $C$ be as defined in Lemma \ref{mixing-time-lemma2}. Observe that if $2^{-j}>\frac{2C\ln n}{\epsilon^3 n}$, then $$-j> -\log_2\left(\frac{\epsilon^3 n}{2C\ln n}\right),$$
and hence $j<\log_2n$. As such, we may define $J$ to be the set of indices $j$ satisfying $2^{-j}\le \frac{2C\ln n}{\epsilon^3n}$ and note that $\big|J^C\big|<\log_2n$. Now, for $i\in J^C$, we have that $2^{-i}>\frac{2C\ln n}{\epsilon^3n}$, and therefore $\Phi(2^{-i})$ denotes the minimum conductance of connected sets $S$ with $\pi(S)\ge \frac{C\ln n}{\epsilon^3n}$, by our definition. Furthermore, since $j\ge 1$, we only consider sets $S$ such that $\pi(S)\le \frac{1}{2}$. Hence, for $i\in J^C$, we may use Lemma \ref{mixing-time-lemma2} to obtain that \textbf{whp} $\Phi\left(2^{-i}\right)\ge c \frac{\epsilon^2}{\ln\left(\frac{1}{\epsilon}\right)}$. Thus, \textbf{whp}
\begin{align*}
    \sum_{j=1}^{\log_2\pi_{\min}^{-1}}\Phi^{-2}\left(2^{-j}\right)&\le \big|J^C\big|\cdot \left(c \frac{\epsilon^2}{\ln\left(\frac{1}{\epsilon}\right)}\right)^{-2}+\sum_{j\in J}\Phi^{-2}\left(2^{-j}\right)\\
    &\le O\left(\frac{\log n}{\epsilon^4}\right)+\sum_{j\in J}\Phi^{-2}\left(2^{-j}\right).
\end{align*}
In order to bound $\sum_{j\in J}\Phi^{-2}\left(2^{-j}\right)$, observe that since $L_1$ is connected, for every subset $S\subset V(L_1)$ we have:
\begin{align*}
    \Phi(S)=\Phi(S^C)\ge \frac{e(S,S^C)}{4e(L_1)\pi(S)}\ge \frac{1}{4e(L_1)\pi(S)}.
\end{align*}
Recall that the excess of $L_1$ is \textbf{whp} at most $\epsilon^2n$, and therefore by Theorem \ref{FKM-thm} \textbf{whp} $e(L_1)\le 3\epsilon n$. Furthermore, when considering $\Phi\left(2^{-j}\right)$, we restrict ourselves to $\pi(S)\le 2^{-j}$. Thus, \textbf{whp}
\begin{align*}
    \Phi(S)\ge \frac{1}{4e(L_1)\pi(S)}\ge \frac{2^j}{12\epsilon n}.
\end{align*}
Hence, we have that \textbf{whp}
\begin{align*}
    \sum_{j\in J}\Phi^{-2}\left(2^{-j}\right)&\le 2\max_{j\in J}\left\{2^{-2j}\right\}12^2\epsilon^2n^2\\
    &\le 2\cdot 12^2\epsilon^2n^2\left(\frac{2C\ln n}{\epsilon^3 n}\right)^2\\
    &\le C'\frac{\ln^2n}{\epsilon^4},
\end{align*}
completing the proof.
\end{proof}

\end{document}